\documentclass[11pt]{amsart}
\usepackage{amsmath, amsthm,amssymb, amsfonts}

\newcommand{\ben}{{\mathbb N}}
\newcommand{\ber}{{\mathbb R}}
\newcommand{\beq}{{\mathbb Q}}

\newtheorem{theorem}{Theorem}[section]
\newtheorem{lemma}[theorem]{Lemma}

\theoremstyle{definition}
\newtheorem{definition}[theorem]{Definition}
\theoremstyle{remark}

\numberwithin{equation}{section}

\theoremstyle{definition}

\begin{document}
\title{Preservation of notion of large sets near zero over reals}
\author{Kilangbenla Imsong}
\address{Department of Mathematics, Nagaland University, Lumami-798627, Nagaland, India.}
\email{kilangbenla@nagalanduniversity.ac.in}
\author{Ram Krishna Paul}
\address{Department of Mathematics, Nagaland University, Lumami-798627, Nagaland, India.}
\email{rmkpaul@gmail.com, rkpaul@nagalanduniversity.ac.in}
\begin{abstract}
    The study of the size of subsets in a semigroup have shown that many of these subsets have strong combinatorial properties and contribute richly to the algebraic structure of the Stone- Cech compactification of a discrete semigroup. N. Hindman and D. Strauss have proved that if $u,v \in \, \ben$, $M$ is a $u \times v$  matrix satisfying restrictions that vary with the notion of largeness and if $\Psi$ is a notion of large sets in $\ben$ then $\{\vec{x} \in \ben^v: M\vec{x} \in \Psi^u \}$ is a large set in $\ben^v$.  In this article, we investigate the above result for various notions of largeness near zero in $\ber^+$.
\end{abstract}
\maketitle

\section{Introduction}
Many of the powerful combinatorial results of Ramsey Theory  can be expressed using matrices. For example, van der Waerden's theorem \cite{W} can be stated as, "The entries of the matrix  $M \Vec{x}$, where 
  \begin{center} 
 $   M = \begin{pmatrix}
      1 & 0\\
      1 & 1\\
      1 & 2\\
      \vdots\\
      1 & l\\
      \end{pmatrix}$
 and 
      $\Vec{x} =\begin{pmatrix}
  a \\
  d
  \end{pmatrix} \in  \mathbb{N}^2$, \\
  \end{center}
  are monochromatic".  This property of the matrix is called \textit{Image Partition Regularity} which is defined below.

  \begin{definition} Let $u,v \in \mathbb{N}$ and let $M$ be a $u \times v$ matrix with entries from $\mathbb{Q}$. The matrix $M$ is image partition regular over $\mathbb{N}$ (abbreviated IPR/$\ben$) if and only if whenever $r \in \mathbb{N}$ and $\mathbb{N} = \bigcup_{i=1}^{r} C_i$, there exists $i\in \{1,2,...,r\}$ and $\Vec{x} \in \mathbb{N}^v$ such that $M\Vec{x} \in C_i^u$ . 
  \end{definition}

The first characterisations of matrices with entries from $\mathbb{Q}$ that are IPR/$\mathbb{N}$ were obtained in \cite{HLb} and over the years other characterisations have been obtained, which can be found in detail in \cite{HSc}.  Among these, the ones involving the notion of \textit{central} sets have been a focal point of study primarily because of the rich combinatorial structure they possess.   Central subsets of $\mathbb{N}$ were first introduced by Furstenberg defined in terms of notions of topological dynamics and he has given and proved the Central Sets Theorem which we state below. We let $\mathcal{P}_f(X)$ to denote any finite collection of subsets of a set $X$.
\begin{theorem} Let $l \in \mathbb{N}$ and for each $i \in \{1,2, \cdots, l\}$, let $\{y_{i,n}\}_{n=1}^{\infty}$ be a sequence in $\mathbb{Z}$. Let $C$ be central subset of $\mathbb{N}$. Then there exists sequences $\{a_n\}_{n=1}^{\infty}$ in $\mathbb{N}$ and $\{
H_n\}_{n=1}^{\infty}$ in $\mathcal{P}_f(\mathbb{N})$ such that 

(1) for all $n$, max $H_n<$ min $H_{n+1}$ and

(2) for all $F \in \mathcal{P}_f(\mathbb{N})$, and all $i \in \{1,2,...,l\}$,
\begin{center}
    $\sum_{n \in F} ( a_n + \sum_{t \in H_n}y_{i,t})\in C$.
\end{center}
\end{theorem}
\begin{proof}
 \cite[Proposition 8.21]{F}
\end{proof}
Central sets were also defined using the algebra of Stone $\check{C}$ech compactification of a discrete semigroup which can be found in  \cite[Definition 4.42]{HSc}. The equivalence of the notions of \textit{central} and \textit{dynamically central} was established by V. Bergelson and N. Hindman in \cite{BH}. 
The following theorem is the characterisation of image partition regular matrices involving central sets.

\begin{theorem} \label{th1}Let $u,v \in \mathbb{N}$ and let $A$ be a $u \times v$ matrix with entries from $\mathbb{Q}$. The following statements are equivalent.

(a) $A$ is image partition regular over $\ben$.

(b) For every central set $C$ in $\mathbb{N}$, there exists  $\Vec{x} \in \mathbb{N}^v$ such that $A\Vec{x} \in C^u$. 

(c)  For every central set $C$ in $\mathbb{N}$, \{$\Vec{x} \in \mathbb{N}^v$ :  $A\Vec{x} \in C^u$ \} is central in $ \mathbb{N}^v$.
    
\end{theorem}
\begin{proof} \cite[Theorem 15.24]{HSc}
\end{proof}

The sets which satisfy the conclusion of Central Sets Theorem are called C-sets and hence the natural question arises as to whether Theorem \ref{th1}  holds true if \textit{central set} is  replaced by \textit{C set}. This has been answered in affirmative and Hindman and Strauss has proved it in \cite [Theorem 1.4]{HSb}.\\
For any dense subsemigroup  $S$ of $(\mathbb{R},+)$, it has been observed that there are subsets of $\mathbb{R}$ living near zero which also satisfy a version of the Central Sets Theorem. These sets have been defined as \textit{central sets near zero} in \cite{HL} and the corresponding notion of \textit{Image Partition Regularity Near Zero} has been developed in \cite{DH}.
\begin{definition}\label{thtB} Let $S$ be a subsemigroup
of $({\mathbb R},+)$ with $0\in$ $c\ell S$, let $u,v\in\ben$, and let
$A$ be a $u\times v$ matrix with entries from $\beq$.
Then $A$ is {\em image partition regular over $S$ near zero\/}
(abbreviated IPR/$S_0$)
if and only if, whenever $S\setminus\{0\}$ is finitely
colored and $\delta>0$, there exists $\vec x\in S^v$
such that the entries of $A\vec x$ are monochromatic and lie
in the interval $(-\delta,\delta)$.
\end{definition}
In the same paper \cite{DH}, image partition regularity near zero over various dense subsemigroups have been considered. Of notable significance for our work is the result that for finite matrices, image partition regular matrices over $\mathbb{R}$ and image partition regular matrices near zero over $\mathbb{R}^+$ are same which has been shown by one of the authors  in \cite [Theorem 2.3]{DPb}. 
The notion of C sets have also been extended to C sets near zero in \cite{BT} and these sets also have rich combinatorial structure. In \cite{KR}, the authors of this paper have proved the following theorem.

\begin{theorem}\label{th2} Let $u,v \in \mathbb{N}$ and let $M$ be a $u \times v$ matrix with entries from $\mathbb{R}$. The following statements are equivalent.\\
\hspace{1 cm}(a) $M$ is image partition regular over $\mathbb{R}^+$.\\
\hspace{1 cm} (b) For every C-set near zero $C$ in $\mathbb{R}^+$, there exists  $\Vec{x} \in (\mathbb{R}^+)^v$ such that $M\Vec{x} \in C^u$. \\
\hspace{1 cm} (c)  For every C-set near zero $C$ in $\mathbb{R}^+$, \{$\Vec{x} \in (\mathbb{R}^+)^v$ :  $M\Vec{x} \in C^u$ \} is a C-set near zero in $ (\mathbb{R}^+)^v$.
\end{theorem}

There are several other notions of largeness in a commutative semigroup and the relationship among them can be found in \cite{Ha}. These notions can also be defined analogously for a dense subsemigroup of $((0,\infty),+)$ and in section 2, we define the notions of  strongly central near zero, strongly central$^*$ near zero, thick near zero, piecewise syndetic$^*$ near zero, $Q^*$ near zero, $P^*$ near zero, central$^*$ near zero, Quasicentral near zero, Quasicentral$^*$ near zero, IP$^*$ near zero, $C^*$ near zero and $J^*$ near zero. If $\Psi$ is one of these notions , in section 3 we prove the following theorem.
\begin{theorem}\label{T1}
    Let $u,v \in \ben, M$ be $u \times v$ matrix with entries from $\ber^+ \cup \{0\}$. If $M$ is $IPR/\ber^+$, $C \subseteq \ber^+$ is a $\Psi$ set in $\ber^+$, then $\{\vec{x} \in (\ber^+)^v: M\vec{x} \in C^u\}$ is a $\Psi$ set in $(\ber^+)^v$.
\end{theorem}
\section{Preliminary Results}
 
For this paper, we need the algebraic structure of the Stone-Cech Compactification, denoted by $\beta S$, of a discrete semigroup $(S,+)$. We let $\beta S=\{ p: p$ is an ultrafilter on $S\}$ and identify the principal ultrafilters on $S$ with the points of $S$ and so we may assume $S \subseteq \beta S$. Given $A \subseteq S$, $\overline{A}=\{p \in \beta S: A \in p\}$, where $\overline{A}$  denotes the closure of $A$ in $\beta S$ and $\{\overline{A}: A \subseteq S\}$ is a basis for the topology of $\beta S$. The operation $+$ on $S$ extends to an operation, also denoted by $+$, on $\beta S$ so that $\beta S$ is a right topological semigroup with $S$ contained in the topological centre of $\beta S$. That is, for each $p \in \beta S$, the function $\rho_p: \beta S \rightarrow \beta S$ defined by $\rho_p(q)=q+p$ is continuous and for each $x \in S$, the function $\lambda_x: \beta S \rightarrow\beta S$ defined by $\lambda_x(q)=x+q$ is continuous. Given $p,q \in \beta S, A \in p+q$ if and only if $\{x \in S: -x+A \in q\} \in p$, where $-x+A=\{y \in S: x+y \in A\}$. Since $\beta S$ is a compact Hausdorff right topological semigroup, it contains idempotents and a smallest two sided ideal, denoted by $K(\beta S)$, which is the union of all of the minimal left ideals of $\beta S$ and also the union of all the minimal right ideals of $\beta S$. The detailed algebraic structure of $\beta S$ is available in \cite{HSc}.
In this section, we prove some results for subsets $E$ of $S^v, v \ge 1$ where $S$ is a dense subsemigroup of $((0, \infty),+)$. We consider $S$ to be a dense subsemigroup of $((0,\infty),+)$, where $\textit
{dense}$ means with respect to usual topology on $((0,\infty),+)$ and when we use the Stone Cech Compactification of such a semigroup $S^v, v \ge 1$ then we deal with $S_d^v$ which is the set $S^v$ with the discrete topology. We also prove some results which will be useful to prove Theorem \ref{T1}.\\
We start with the following definition.
\begin{definition} \label{defA} Let $S$ be a dense subsemigroup of $((0,\infty),+)$ and $v \ge 1$. Then 
$0^+(S^v)=\{ p\in\beta (S_d^v): (\forall \epsilon>0)(((0,\epsilon) \cap S)^v \in p)\}$.  
\end{definition}
The following is Lemma 2.2 of \cite{KR}.
\begin{lemma} \label{l} Let $S$ be a dense subsemigroup of $((0,\infty),+)$, then $0^+(S^v)$, $v \ge 1$ is a compact right topological subsemigroup. 
\end{lemma}
As a consequence of Lemma \ref{l}, $0^+(S^2)$  contains idempotents. The follwing lemma is a very useful result and can be found in \cite[Lemma 2.4]{KR}. 
\begin{lemma}
     Let $S$ be a dense subsemigroup of $((0,\infty),+)$ then for $v \ge 1$,\\  $0^+(S^v) \bigcap K(\beta S^v) = \phi$.
\end{lemma}
Let $S$ be a dense subsemigorup of $((0,\infty),+)$. Let $A \subseteq S^v$ for $v \ge 1$. The set of sequences in $S^v$ converging to zero is denoted by $\mathcal{T}_0^v$.
 In the following definitions, $\overline{A}$ denotes the closure of $A$ in $0^+(S^v)$ and $E(A)$ denotes the set of idempotents in $A$. We recall the following definitions from \cite{BT} and \cite{DP}.
\begin{definition}  
\begin{itemize}
     \item [(1)]  $A$ is an IP set near zero if and only if there exists an idempotent $p \in 0^+(S^v)$ such that $A \in p$. 
     \item [(2)] $A$ is $IP^*$ set near zero if and only if $E(0^+(S^v)) \subseteq \overline{A}$.
     \item [(3)] $A$ is a  central$^*$ set near zero if and only if $E(K(0^+(S^v))) \subseteq \overline{A}$.
     \item[(4)] $A$ is a J-set near zero in $S^v$ iff whenever $F \in \mathcal{P}_f(\mathcal{T}_0^v)$ and $\delta > 0$, there exists $\vec{a} \in S^v \cap (0,\delta)^v$ and $H \in \mathcal{P}_f(\mathbb{N})$ such that for each $\vec{f} \in F,\, \vec{a} +  \sum_{t \in H}\vec{f}(t) \in A$.    
    \item[(5)] $J_0(S^v)=\{ p\in 0^+(S^v): \forall A \in p, A $ is a J-set near zero \}.     
\end{itemize}
\end{definition}

We now define some notions of largeness near zero.
\begin{definition} 
\begin{itemize}
 \item [(1)] $A$ is a $P$ set near zero in $S^v$ if and only if for every $k \in \mathbb{N}, \epsilon>0, A \cap (0, \epsilon)^v$ contains a length $k$ Arithmetic Progression(AP).
\item [(2)] $Prog(S_0^v)=\{ p \in 0^+(S^v): \forall \; A \in p, A$ is a P set near zero\}.
 \item [(3)]  $A$ is a Q set near zero if and only if there exists a sequence $\{\vec{x}_n\}_{n=1}^\infty$ in $S^v$ such that $ \sum_{n=1}^\infty \vec{x}_n$ converges and whenever $m<n, \vec{x}_n \in \vec{x}_m +A$.
 \item [(4)] $A$ is a strongly central $(SC)$ set near zero if and only if for every minimal left ideal $L$ of $0^+(S^v)$ there is an idempotent in $L\cap \overline{A}$.  
 \item [(5)] $A$ is a strongly central$^*(SC^*)$ set near zero if and only if there is a minimal left ideal $L$ of $0^+(S^v)$ with $E(L) \subseteq \overline{A}$.   
 \item [(6)]$A$ is a thick set near zero if and only if there is a left ideal $L$ of $0^+(S^v)$ such that  $L \subseteq \overline{A}$.
 \item[(7)]$A$ is a Piecewise Syndetic$^*$(PS$^*)$ set near zero if and only if $K(0^+(S^v)) \subseteq \overline{A}$.
 \item [(8)] $A$ is $C^*$ set near zero if and only if $E(J_0(S^v)) \subseteq \overline{A}$. 
 \item [(9)] $A$ is a Quasi Central(QC) set near zero if and only if there is an idempotent in $Cl_{0^+(S^v)} K(0^+(S^v)) \cap \overline{A}$.
  \item [(10)] $A$ is a Quasi Central$^*$(QC$^*)$ set near zero if and only if \sloppy $E(Cl_{0^+(S^v)}K(0^+(S^v))) \subseteq \overline{A}$.     
\item [(11)] $A$ is $J^*$ set near zero if and only if $J_0(S^v) \subseteq \overline{A}$. 
\item [(12)] $A$ is a $Q$* set near zero or a $P$* set near zero if and only if $S^v \setminus A$ is not a $Q$ set near zero or not a $P$ set near zero respectively.
    \end{itemize}
    
\end{definition}
The notions of P set and Q set were proved to be partition regular in \cite{Ha}. We prove in the following theorems that their analogous counterparts, P set near zero and Q set near zero are also partition regular. 
\begin{theorem} Let $S$ be a dense subsemigorup of $((0,\infty),+)$. The property P set near zero is partition regular. Consequently, if $A\subseteq S$ and A is a P set near zero then $\overline{A} \cap Prog(S_0) \neq \emptyset$. $Prog(S_0)$ is a compact two sided ideal of $0^+(S)$.  
\end{theorem}
\begin{proof}
    Let $A_1 \cup A_2$ be a P set near zero. We will show there exists $i \in \{1,2\}$ such that $A_i$ is a P set near zero, i.e., for each $k \in \ben, \epsilon>0$, there exists $i(k) \in \{1,2\}$ such that $A_{i(k)} \cap (0, \epsilon)$ contains a length $k$ arithmetic progression. By Vander Waerden's theorem, \cite[Corollary 14.3]{HSc}, there exists $n \in \ben$ such that whenever $\{1,2,\cdots,n\}$ is 2-coloured there is a monochromatic length $k$ arithmetic progression. Since $A_1 \cup A_2$ is a P set near zero, for $n \in \ben$ and  $\epsilon > 0$ there exists $ a, d \in S$ such that $\{a+td: t \in \{1,2,\cdots,n\}\} \subseteq (A_1 \cup A_2) \cap (0,\epsilon)$. Define for each $i \in \{1,2\}, B_i=\{ t \in \{1,2,\cdots,n\}: a+td \in A_i \cap (0,\epsilon)\}$. Let $i(k) \in \{1,2\}$ such that $B_{i(k)}$ contains length $k$ arithmetic progression. Then there exists $e,f \in S$ such that $\{e+sf:s\in \{1,2,\cdots,k\}\} \subseteq B_{i(k)}$. Then for $s \in \{1,2,\cdots,k\}, a+(e+sf)d \in A_{i(k)} \cap (0, \epsilon)$. Let $a'=a+ed, d'=fd$ then, $\{a'+sd': s \in \{1,2,\cdots,k\}\} \subseteq A_{i(k)} \cap (0,\epsilon)$.
    
     Let $\mathcal{R}=\{A \subseteq S: A$ is a P set near zero$\}, \mathcal{R^\uparrow}=\{B \subseteq S: A \subseteq B$ for some $A \in \mathcal{R}\}$. Then by \cite[Theorem 3.11]{HSc}, there exists $q \in 0^+(S)$ such that $A \in q \subseteq \mathcal{R^\uparrow}$ and so $q \in Prog(S_0)$. Thus $Prog(S_0) \neq \emptyset$. We will now prove that $Prog(S_0)$ is a closed subset of $0^+(S)$. Let $q \in 0^+(S) \setminus Prog(S_0)$. Then there exists $B \in q$ such that $B$ is not a P set near zero. If $\overline{B} \cap Prog(S_0) \neq \emptyset$, there exists $r \in Prog(S_0)$ such that $B \in r$  which implies $B$ is a P set near zero which is a contradiction. Therefore $q$ is not a limit point of $Prog(S_0)$. Thus $Prog(S_0)$ is a compact set in $0^+(S)$.
   
    We now prove that $Prog(S_0)$ is a two sided ideal of $0^+(S)$. Let $p \in Prog(S_0), q \in 0^+(S), A \in q+p, k \in \ben, \epsilon >0$. We will prove $A \cap (0,\epsilon)$ contains a length $k$ arithmetic progression. Since $A \in q+p \implies A \in \rho_p(q)$ and since $\rho_p$ is continuous, there exists $B \in q$ such that $\rho_p(\overline{B}) \subseteq \overline{A}$. Since $q \in 0^+(S)$, for all $\epsilon > 0, (0,\epsilon) \cap B \in q$. So pick $x \in B$ such that $x < \epsilon$. Then $\lambda_x(p)=x+p=\rho_p(x) \in \overline{A}$. Pick $C \in p$ such that $\lambda_x(\overline{C}) \subseteq \overline{A}$, i.e., $\lambda_x (C)\subseteq A$. Since $C \in p$, C is a P set near zero and so for $\epsilon'=\epsilon-x, C \cap (0,\epsilon')$ contains a length $k$ arithmetic progression. Choose $a,d \in S$ such that  $\{a+td: t \in \{1,2,\cdots,k\}\} \subseteq C \cap (0,\epsilon')$. For all $t \in \{1,2,\cdots,k\}, x+a+td \in x+C \subseteq A$ and $x+a+td \in (x,x+\epsilon') \subseteq (0,\epsilon)$. Let $x+a=a'$, then $\{a'+td: t \in \{1,2,\cdots,k\}\} \subseteq A \cap (0,\epsilon)$. \\
    To prove $Prog(S_0)$ is a right ideal of $0^+(S)$, let $p \in Prog(S_0), q \in 0^+(S), A \in p+q, k \in \ben, \epsilon >0$. We will prove $A \cap (0,\epsilon)$ contains length $k$ arithmetic progression. Since $A \in p+q \implies A \in \rho_q(p)$ and since $\rho_q$ is continuous, pick $B \in p$ such that $\rho_q(\overline{B}) \subseteq \overline{A}$. Now $B$ is a P set near zero so $B \cap (0,\frac{\epsilon}{2})$ contains a length $k$ arithmetic progression. Pick $a, d \in S$ such that  $\{a+td: t \in \{1,2,\cdots,k\}\} \subseteq B \cap (0, \frac{\epsilon}{2})$. For each $t \in \{1,2,\cdots,k\}$, let $x(t)=a+td$, then $x(t) \in B \cap (0, \frac{\epsilon}{2})$. Now, $\lambda_{x(t)}(q)=x(t)+q=\rho_q(x(t)) \in \overline{A}$. Since $\lambda_{x(t)}$ is continuous, pick $C_t \cap (0,\frac{\epsilon}{2})\in q$ such that $\lambda_{x(t)}(\overline{C_t \cap (0,\frac{\epsilon}{2})}) \subseteq \overline{A}$. Let $z \in \bigcap_{t=1}^k (C_t \cap (0,\frac{\epsilon}{2}))$, then $x(t) +z \in A$, i.e., $a+td+z \in A$ for all $t \in \{1,2,\cdots,k\}$. Now, $a+td+z \in (0,\epsilon)$. Let $a'=a+z$, then $\{a'+td: t \in \{1,2,\cdots,k\}\} \subseteq A \cap (0, \epsilon)$.
    
\end{proof}

\begin{theorem}
 Let $S$ be a dense subsemigorup of $((0,\infty),+)$. The property $Q$ set near zero is partition regular.   
\end{theorem}
\begin{proof}
    Let $A_1, A_2 \subseteq S$ such that $A_1 \cup A_2$ is a Q set near zero. We will show there exists $i \in \{1,2\}$ such that $A_i$ is a Q set near zero. Let $\{x_n\}_{n=1}^\infty$ be a sequence in $S$ such that $ \sum_{n=1}^\infty x_n$ converges and whenever $m<n, x_n \in x_m +(A_1 \cup A_2)$. Let $N'=\{m,n \in \ben: m<n    \; \text{and} \; x_n \in x_m +(A_1 \cup A_2)\}$. then $N'$ is an infinite subset of $\ben$. Let $[N']^2=B_1 \cup B_2$ where for each $i \in \{1,2\}, B_i=\{\{m,n\}: m,n \in \ben, m<n \; \text{and} \; x_n \in x_m+A_i\} $. Using Ramsey's Theorem, there exists $i \in \{1,2\}$ and an infinite subset $A$ of $N'$ such that $\{\{m,n\}: m,n \in A, m\neq n\} \subseteq B_i$. Write $A$ in increasing order as $A=\{t_k\}_{k=1}^\infty$. For any $t_k < t_{k'}, x_{t_{k'}} \in x_{t_k}+A_i$. For $k \in \ben$, take $y_k=x_{t_k}$, then $\sum_{k=1}^\infty{y_k}$ converges and for any $m < n, y_n \in y_m + A_i$. 
\end{proof}

\section{preservation of notions of large sets near zero}
As was mentioned in section 1, in this section we will prove Theorem \ref{T1} for various notions of large sets near zero. We begin with the following theorem which follows easily from \cite[Theorem 3.1]{HL}
\begin{theorem} Let $S$ be a dense subsemigroup of $((0,\infty),+)$ and let $A \subseteq S$. Then there exists $p=p+p$ in $0^+(S)$ with $A \in p$ if and only if for $\epsilon > 0$, there is some sequence $\{x_n\}_{n=1}^\infty$ in $(0,\epsilon)$ such that $\sum_{n=1}^\infty x_n$ converges and $FS(\{x_n\}_{n=1}^\infty) \subseteq A$.
\end{theorem}
\begin{proof}
    A sequence $\{x_n\}_{n=1}^\infty$ can be chosen in $(0,\epsilon)$ in the proof of \cite[Theorem 3.1]{HL}.
\end{proof}
The following theorems, Theorems \ref{th3} and \ref{th4}, are needed to prove Theorem \ref{T1} which are slightly stronger versions of Theroem 2.2 and Theorem 2.3 of \cite{DPb}. 
\begin{theorem} \label{th3}
    Let $u,v \in \, \ben, M$ be a $u \times v$ first entries  matrix with entries from $\ber$ and let $C$ be central set near-zero in $\ber^+$. Let $\epsilon > 0$, then there exists sequences \sloppy $\{x_{1,n}\}_{n=1}^\infty, \{x_{2,n}\}_{n=1}^\infty, \cdots, \{x_{v,n}\}_{n=1}^\infty$ in $\ber^+$ such that for any $i \in \{1,2,\cdots,v\}, \sum_{t=1}^{\infty} x_{i,t}$ converges and for each \sloppy $F \in \mathcal{P}_f(\mathbb{N}), M \vec{x}_F \in C^u$, where $\vec{x}_F=\begin{pmatrix}
     \sum_{t \in F} x_{1,t}\\
     \vdots\\
     \sum_{t \in F} x_{v,t}
    \end{pmatrix}$ $\in (0,\epsilon)^v$. 
\end{theorem}
\begin{proof}
    By \cite[Lemma 2.2]{DPb}, there exists sequences $\{x_{1,n}\}_{n=1}^\infty, \{x_{2,n}\}_{n=1}^\infty, \cdots, \{x_{v,n}\}_{n=1}^\infty$ in $\ber^+$ satisfying the given hypotheses. We can choose subsequences, say, \sloppy $\{y_{1,n}\}_{n=1}^\infty, \{y_{2,n}\}_{n=1}^\infty, \cdots, \{y_{v,n}\}_{n=1}^\infty$ of \sloppy $\{x_{1,n}\}_{n=1}^\infty, \{x_{2,n}\}_{n=1}^\infty, \cdots, \{x_{v,n}\}_{n=1}^\infty$ respectively such that for every $F \in \mathcal{P}_f(\ben)$ and for each $i \in  \{1,2, \cdots, v \}$,$\sum_{t \in F} y_{i,t} \in (0,\epsilon)$ . Take $\vec{x}_F=\begin{pmatrix}
     \sum_{t \in F} y_{1,t}\\
     \vdots\\
     \sum_{t \in F} y_{v,t}
    \end{pmatrix}$. Then $\vec{x}_F \in (0,\epsilon)^v$ such that $M \vec{x}_F \in C^u$.
\end{proof}

\begin{theorem} \label{th4}
    Let $M$ be a $u \times v$ matrix with entries from $\ber$. Then the following are equivalent.\\
    1) $M$ is $IPR/\ber^+$.\\
    2)$M$ is $IPR$ near zero over $\ber^+$.\\
    3) For any central set $C$ in $\ber^+$, there exists $\vec{x} \in (\ber^+)^v$ such that $M \vec{x} \in C^u$.\\
    4) For any central set near zero $C$ in $\ber^+$, there exists $\vec{x} \in (\ber^+)^v$ such that $M \vec{x} \in C^u$.\\
    5) For any central set near zero $C$ in $\ber^+$ and $\epsilon>0$, there exists $\vec{x} \in (0, \epsilon)^v$ such that $M \vec{x} \in C^u$.
\end{theorem}
\begin{proof}
    The equivalence of (1),(2),(3) and (4) follows from \cite[Theorem 2.3]{DPb}. (5) implies (1) is trivial. We now prove (1) implies (5). Let $C$ be a central set near zero in $\ber^+$ and $M$ be $IPR/\ber^+$. So by \cite[Theorem 4.1]{H}, there exists $m \in \{1,2,\cdots,u\}$ and a $v \times m$ first entries matrix $G$ with non-negative entries and no row equal to $\vec{0}$ such that, if $B=MG$ then $B$ is a first entries matrix with all of its first entries equal to 1. Let $G=(g_{ij})_{v \times m}$ and $\epsilon>0$ and let $g$=max$\{g_{ij}|1 \le i \le v, 1 \le j \le m \}$. By Theorem \ref{th3}, there exists $\vec{y} \in (0, \frac{\epsilon}{gm})^m$ such that $B\vec{y} \in C^u$, i.e., $MG\vec{y} \in C^u$. Now, $G\vec{y}= \begin{pmatrix}
        g_{11} \cdots g_{1m}\\
        \vdots \\
        g_{v1} \cdots g_{vm}
    \end{pmatrix} \begin{pmatrix}
        y_1\\
        \vdots \\
        y_m
    \end{pmatrix} = \begin{pmatrix}
        g_{11}y_1 + \cdots +g_{1m}y_m\\
        \vdots \\
        g_{v1}y_1 + \cdots + g_{vm}y_m
\end{pmatrix}.$
Now, for each $i \in \{1,2,\cdots,v\}, g_{i1}y_1 + \cdots +g_{im}y_m \le g\frac{\epsilon}{gm} + \cdots + g \frac{\epsilon}{gm} =\epsilon$\\
Take $\vec{x}=G\vec{y} \in (0,\epsilon)^v$ . Therefore, there exists $\vec{x} \in (0,\epsilon)^v$ such that $M\vec{x} \in C^u$. 
\end{proof}
The following results, lemma 4.4, lemma 4.6 and lemma 4.7 are very often used in the theorems that follow.

\begin{lemma} \label{l1} Let $u,v \in \ben$ and let $M$ be a $u \times v$ matrix with entries from $\ber^+ \cup \{0\}$ and with no row equal to the zero row. Define $\phi : (\ber ^+)^v \longrightarrow (\ber^+)^u$ by $\phi(\vec {x}) = M\vec {x}$ and let $\widetilde{\phi}: \beta ((\ber^+)^v)\longrightarrow (\beta \ber^+)^u$ be its continuous extension.
Let $q \in 0^+((\ber^+)^v)$. Then $\widetilde{\phi}(q) \in (0^+(\ber^+))^u$
\end{lemma}
\begin{proof}
    Let $q \in 0^+((\ber^+)^v)$ and let $\Tilde{\phi}(q)= \begin{pmatrix}
    p_1\\
    \vdots\\
    p_u
\end{pmatrix}$. We have to prove that for each $i \in \{1,2, \cdots, u\}, (0,\epsilon) \in p_i \; \forall \; \epsilon > 0$. Let $i \in \{1,2, \cdots, u\}$ and suppose there exists some $\epsilon_i >0$ such that $(0,\epsilon_i) \notin p_i$. Then $\ber^+ \setminus (0,\epsilon_i) \in p_i$. Let for each $i=\{1,2,\cdots,u\}$, $C_i=\ber^+ \setminus (0,\epsilon_i)$ . Then $C_i \in p_i$ for each $i \in \{1,2,\cdots, u\}$. Let $\delta =$ min $\{\epsilon_1, \cdots, \epsilon_u\}$ and $C=\ber^+ \setminus(0,\delta)$. Then $C \in p_i$ for all $i \in \{1,2,\cdots, u\}$. Thus, $p_i \in C\ell_{\beta \ber^+}{C}$ for each $i \in \{1,2,\cdots, u\}$ and so, $C\ell_{(\beta \ber^+)^u}{C}^u$ is a neighbourhood of $\widetilde{\phi}(q)$. Since $\widetilde{\phi}$ is continuous, there exists a neighbourhood, say, $C\ell_{\beta (\ber^+)^v}W$ of $q$, where $W \in q$, such that $\widetilde{\phi}(C\ell_{\beta (\ber^+)^v}{W}) \subseteq C\ell_{(\beta \ber^+)^u}{C}^u \implies \phi(W) \subseteq C^u \implies W \subseteq \phi^{-1}[C^u] \implies \phi^{-1}[C^u] \in q$. Therefore, for any $\epsilon >0, (0,\epsilon)^v \cap \phi^{-1}[C^u] \in q$. This implies that for any $\epsilon >0$, there exists $ \vec{x} \in (0,\epsilon)^v$ such that $ \phi(\vec{x}) \in C^u$. That is, for any $\epsilon >0$, there exists $ \vec{x} \in (0,\epsilon)^v$ such that every entry of $ \phi(\vec{x}) =M\vec{x}$ is in $[\delta,\infty)$. 

Let $M=(a_{ij})$ and $a$= max $\{a_{ij}| 1 \le i \le u, 1 \le j \le v\}$ then $a>0$. Choose $\epsilon = \frac{\delta}{va}$. Then for any $\vec{x} \in (0,\frac{\delta}{va})^v, \vec{x}=(x_1, \cdots, x_v), a_{i1}x_1+\cdots+a_{iv}x_v < a \frac{\delta}{va}+ \cdots+a\frac{\delta}{va} =\delta$ for all $i \in \{1,2,\cdots, u\} $,i.e., $M\vec{x} \notin C^u$. This is a contradiction. Therefore, $\widetilde{\phi}(q) \in (0^+(\ber^+))^u$.
\end{proof}

\begin{lemma} \label{l2} Let $u,v \in \ben$ and let $M$ be a $u \times v$ matrix with entries from $\ber^+ \cup \{0\}$ and with no row equal to the zero row. Define $\phi : (\ber ^+)^v \longrightarrow (\ber^+)^u$ by $\phi(\vec {x}) = M\vec {x}$ and let $\widetilde{\phi}: \beta ((\ber^+)^v)\longrightarrow (\beta \ber^+)^u$ be its continuous extension. Let $r \in (\beta (\ber^+)^v) \setminus (0^+((\ber^+)^v))$. Then $\widetilde{\phi}(r) \notin (0^+(\ber^+))^u$.
\end{lemma}
\begin{proof}
    Let $r \in (\beta (\ber^+)^v) \setminus (0^+((\ber^+)^v))$. Then there exists $\epsilon >0$ such that $(0,\epsilon)^v \notin r$, i.e., $[\epsilon,\infty)^v \in r$. Suppose $\widetilde{\phi}(r) \in (0^+(\ber^+))^u$. Let $\Tilde{\phi}(r)= \begin{pmatrix}
    r_1\\
    \vdots\\
    r_u
\end{pmatrix}$ $\in (0^+(\ber^+))^u$. 
Let $M=(a_{ij})$ and $a$= min $\{\sum_{j=1}^v a_{ij}: i=1,2,\cdots, u\}$. Then $a>0$.
 Since $\widetilde{\phi}$ is continuous, for the neighbourhood $C\ell_{0^+(\ber^+)}(B_1 \cap (0,\epsilon a)) \times \cdots \times C\ell_{0^+(\ber^+)}(B_u \cap (0,\epsilon a))$ of  $\widetilde{\phi}(r)$, where $B_i \in r_i$ for $i=1,2,\cdots,u$ and $W \in r$,  there exists a neighbourhood $C\ell_{\beta (\ber^+)^v}W$ of $r$, such that   $\widetilde{\phi}(C\ell_{\beta (\ber^+)^v)}{W}) \subseteq C\ell_{0^+(\ber^+)}(B_1 \cap (0,\epsilon a)) \times \cdots \times C\ell_{0^+(\ber^+)}(B_u \cap (0,\epsilon a))$, i.e., $\phi(W) \subseteq B_1 \cap (0,\epsilon a) \times \cdots \times B_u \cap (0,\epsilon a)$. Now $W \cap [\epsilon, \infty)^v \in r$, therefore $\phi(W \cap [\epsilon, \infty)^v) \subseteq B_1 \cap (0,\epsilon a) \times \cdots \times B_u \cap (0,\epsilon a)$. Take $\vec{z} =\begin{pmatrix}
     z_1\\
     \vdots\\
     z_v
 \end{pmatrix} \in (W \cap [\epsilon, \infty)^v)$. Then $\phi(\vec{z})=M\vec{z} =\begin{pmatrix}
     a_{11}z_1 \cdots a_{1v}z_v\\
     \vdots\\
     a_{u1}z_1 \cdots a_{uv}z_v
 \end{pmatrix}$ . Now, for each $i \in \{1,2,\cdots,u\}, a_{i1}z_1+ \cdots +a_{iv}z_v \ge a_{i1}\epsilon+\cdots+a_{iv}\epsilon =(a_{i1} +\cdots+a_{iv})\epsilon>a\epsilon $. Thus, $\phi(\vec{z}) \notin (0,\epsilon a)^v$. This is a contradiction.
\end{proof}

\begin{lemma} \label{leA} Let $u,v \in \ben$ and let $M$ be a $u \times v$ matrix with entries from $\ber^+ \cup \{0\}$ and with no row equal to the zero row. Define $\phi : (\ber ^+)^v \longrightarrow (\ber^+)^u$ by $\phi(\vec {x}) = M\vec {x}$ and let $\widetilde{\phi}: \beta ((\ber^+)^v)\longrightarrow (\beta \ber^+)^u$ be its continuous extension.
Let $p$ be an idempotent in $0^+(\ber^+)$. If for every $\epsilon >0$ and every $C \in p$, there exists $\vec {x} \in (0,\epsilon)^v$ such that $ M \vec {x} \in C^u$ then there exists an idempotent $q \in 0^+((\ber^+)^v)$ such that $\widetilde{\phi}(q)=\vec{p}= \begin{pmatrix}
    p\\
    \vdots\\
    p
\end{pmatrix} \in (0^+(\ber^+))^u$. If $p$ is minimal idempotent then $q$ is also minimal idempotent and $\widetilde{\phi}[K(0^+((\ber^+)^v))]=\widetilde{\phi}[0^+((\ber^+)^v)] \bigcap (K(0^+(\ber^+)))^u$.
\end{lemma}
\begin{proof}
    Let $p$ be an idempotent in $0^+(\ber^+)$. Let $\mathcal{A}=\{q \in 0^+((\ber^+)^v): \widetilde{\phi}(q)=\vec{p}\}$. Then we will prove that $\mathcal{A}$ is non-empty, i.e, $\widetilde{\phi}^{-1}(\{\vec{p}\}) \cap 0^+((\ber^+)^v) \neq \emptyset$ 
    i.e., to prove that $\vec{p} \in \widetilde{\phi}[0^+(\ber^+)^v]$. Suppose $\vec{p} \notin \widetilde{\phi}[0^+((\ber^+)^v)]$. Then $\vec{p}$ is not a limit point of $\widetilde{\phi}[0^+((\ber^+)^v)]$. Therefore, there exists a neighbourhood $W$ of $\vec{p}$ such that  $W \bigcap \widetilde{\phi}(0^+((\ber^+)^v)) = \emptyset$, where $W = C\ell_{0^+(\ber^+)}{B}_1 \times \cdots \times C\ell_{0^+(\ber^+))}{B}_u$ for $B_i \in p, i=1,2, \cdots u$. Take $B=B_1 \cap \cdots \cap B_u$ then $B \in p$. Therefore, $C\ell_{(0^+(\ber^+))^u}{B^u} \cap  \widetilde{\phi}(0^+((\ber^+)^v)) = \emptyset$. Given for $B \in p, \epsilon >0$, there exists $\vec {x} \in (0,\epsilon)^v$ such that $ \phi( \vec {x}) \in B^u$, i.e., $\vec{x} \in \phi^{-1}[B^u]$. Therefore, $\phi^{-1}[B^u] \cap (0,\epsilon)^v \neq \emptyset $ for every $\epsilon > 0$. Let $\mathcal{C} =\{C\ell_{0^+((\ber^+)^v)}{\phi^{-1}[B^u]} \cap C\ell_{0^+((\ber+)^v)} {(0,\epsilon)^v}: \epsilon > 0\}$. Then, $\mathcal{C}$ is a collection of closed sets in $0^+((\ber^+)^v)$ which has the finite intersection property. Therefore, $\bigcap_{\epsilon >0}(C\ell_{0^+((\ber^+)^v)}{\phi^{-1}[B^u]} \cap (C\ell_{0^+((\ber^+)^v)} {(0,\epsilon)^v}) \neq \emptyset$, i.e., $(C\ell_{0^+((\ber^+)^v)}{\phi^{-1}[B^u]} \cap (\cap_{\epsilon >0}(C\ell_{0^+((\ber^+)^v)} {(0,\epsilon)^v}) \neq \emptyset$, i.e.,$(C\ell_{0^+((\ber^+)^v)}{\phi^{-1}[B^u]} \cap 0^+((\ber^+)^v) \neq \emptyset$. Let $r \in (C\ell_{0^+((\ber^+)^v)}{\phi^{-1}[B^u]} \cap 0^+((\ber^+)^v)$. This implies that $r \in 0^+((\ber^+)^v)$ and $\widetilde{\phi}(r) \in C\ell_{(0^+(\ber^+))^u}{B^u}$. Therefore,  $C\ell_{(0^+(\ber^+))^u}{B^u} \cap  \widetilde{\phi}(0^+((\ber^+)^v)) \neq \emptyset$, which is a contradiction. Therefore, $\mathcal{A}$ is non-empty. Now, $\mathcal{A}= \widetilde{\phi}^{-1}(\{\vec{p}\})$ is a compact right topological semigroup in $0^+((\ber^+))^v$. Thus, there exists an idempotent $q \in 0^+((\ber^+)^v)$ such that $\widetilde{\phi}(q)=\vec{p}$. 
    
    Let $p$ be minimal idempotent in $0^+(\ber^+)$. Then there exists an idempotent $w$ in $0^+((\ber^+)^v)$ such that $\widetilde{\phi}(w)=\vec{p}$. For this idempotent $w$, by \cite[Theorem 1.60]{HSc} there exists a minimal idempotent, say $q$ in $0^+((\ber^+)^v)$ such that $q \le w$. i.e., $q=q+w \implies \widetilde{\phi}(q)=\widetilde{\phi}(q+w)=\widetilde{\phi}(q)+ \widetilde{\phi}(w) \implies \widetilde{\phi}(q) \le \widetilde{\phi}(w)$. Since $ q$ and $w$ are idempotents and $\widetilde{\phi}$ is homomorphism, therefore, $\widetilde{\phi}(q)$ and $\widetilde{\phi}(w)$ are also idempotents in $(0^+( \ber^+))^u$ by using Lemma \ref{l1}. Since $p$ is minimal, $\vec{p}$ is also minimal and $\widetilde{\phi}(q) \le \vec{p}$ which implies that $\widetilde{\phi}(q) =\vec{p}$. Therefore, there exists minimal idempotent $q \in 0^+((\ber^+)^v)$ such that  $\widetilde{\phi}(q) =\vec{p}$. Therefore, $\widetilde{\phi}[0^+((\ber^+)^v] \bigcap K((0^+(\ber^+))^u) \neq \emptyset$ and by \cite[Theorem  2.23]{HSc},  $\widetilde{\phi}[0^+((\ber^+)^v] \bigcap (K(0^+(\ber^+)))^u \neq \emptyset$.

By\cite[Theroem 1.65]{HSc}, $K[\widetilde{\phi}(0^+((\ber^+)^v)] = \widetilde{\phi}(0^+((\ber^+)^v) \bigcap (K(0^+(\ber^+)))^u$. By \cite[Exercise 1.7.3]{HSc},  $\widetilde{\phi}[K(0^+((\ber^+)^v))]= K[\widetilde{\phi}(0^+((\ber^+)^v)]$. Therefore, \sloppy $\widetilde{\phi}[K((0^+((\ber^+)^v))] = \widetilde{\phi}(0^+((\ber^+)^v) \bigcap (K(0^+(\ber^+)))^u$. 
\end{proof} 

\begin{lemma} \label{leB}
    Let $u,v \in \ben$ and let $M$ be a $u \times v$ matrix with entries from $\ber^+ \cup \{0\}$ and with no row of $M$ as the zero row. Define $\phi : (\ber ^+)^v \longrightarrow (\ber^+)^u$ by $\phi(\vec {x}) = M\vec {x}$ and let $\widetilde{\phi}: \beta ((\ber^+)^v)\longrightarrow (\beta \ber^+)^u$ be its continuous extension.
    Let $p \in 0^+(\ber^+), C \in p$, then $\widetilde{\phi}^{-1}[C\ell _{(0^+(\ber^+))^u} C^u]$=$C\ell _{0^+((\ber^+)^v)} \phi^{-1}[C^u]$. 
\end{lemma}
\begin{proof}
    Since $\widetilde{\phi}$ is continuous, \sloppy $\widetilde{\phi}^{-1}[C\ell _{(0^+(\ber^+))^u} C^u]$ is a closed set  containing $\phi^{-1}[C^u]$. But $C\ell _{0^+((\ber^+)^v)} \phi^{-1}[C^u]$ is smallest closed set containing $\phi^{-1}[C^u]$, therefore $\widetilde{\phi}^{-1}[C\ell _{(0^+(\ber^+))^u} C^u] \supseteq C\ell _{0^+((\ber^+)^v)} \phi^{-1}[C^u]$. \\
    Let  $q \in \widetilde{\phi}^{-1}[C\ell _{(0^+(\ber^+))^u} C^u]$. By Lemma \ref{l2}, $q \in 0^+((\ber^+)^v)$. Now,  $\widetilde{\phi}(q) \in C\ell _{(0^+(\ber^+))^u} C^u$. Since $\widetilde{\phi}$ is continuous, therefore there exists a neighbourhood $C\ell_{0^+((\ber^+)^v)} W$ of $q$ such that $\widetilde{\phi}(C\ell_{0^+((\ber^+)^v)} W) \subseteq C\ell _{(0^+(\ber^+))^u} C^u$ or, $\phi(W) \subseteq C^u$ or, $W \subseteq \phi^{-1}[C^u]$. Since $W \in q$, therefore $\phi^{-1}[C^u] \in q$, i.e., $q \in C\ell _{0^+((\ber^+)^v)} \phi^{-1}[C^u]$.
\end{proof}

We are now in a position to prove Theorem \ref{T1} for the various notions of largeness.

\begin{theorem}
    Let $u,v \in \ben$ and $M$ be a $u \times v $ matrix with entries from $\ber^+ \cup \{0\}$ which is $IPR/\ber^+$. If $C$ is an $SC^*$ set near zero in $\ber^+$, then $\{\vec{x} \in (\ber^+)^v: M\vec{x} \in C^u\}$ is an $SC^*$ set near zero in $(\ber^+)^v$.
\end{theorem}
\begin{proof}
    Let $\phi : (\ber ^+)^v \longrightarrow (\ber^+)^u$ be defined by $\phi(\vec{x}) = M\vec{x}$ and let $\widetilde{\phi}: \beta ((\ber^+)^v)\longrightarrow (\beta \ber^+)^u$ be its continuous extension. Since $C$ is $SC^*$ set near zero, there exists a minimal left ideal $L$ of $0^+(\ber^+)$ with $E(L) \subseteq C\ell_{0^+(\ber^+)}{C}$. Let $p$ be an idempotent in $L$ and let $\vec{p}= \begin{pmatrix}
    p\\
    \vdots\\
    p
\end{pmatrix}$ 
$\in (0^+(\ber^+))^v$. Since every member of $p$ is central near zero, so by Theorem \ref{th4}, $\forall \; C \in p, \epsilon>0, \exists \; \vec{x} \in (0,\epsilon)^v$ such that $M\vec{x} \in C^u$. Also by Lemma \ref{leA}, there exists an idempotent $q \in K(0^+((\ber^+)^v))$ such that $\widetilde{\phi}(q) = \vec {p}$. Let $L'=0^+((\ber ^+)^v) +q$. Then $L'$ is a minimal left ideal of $0^+((\ber ^+)^v)$. We claim that $E(L') \subseteq Cl_{0^+((\ber ^+)^v)} \phi^{-1}[C^u]$. Let $r \in E(L')$. Then $r=r+q$. Let $\widetilde{\phi}(r)=\begin{pmatrix}
    s_1\\ 
    \vdots\\
    s_u
\end{pmatrix}$ 
$\in (0^+(\ber^+))^u$. Since $\widetilde{\phi}$ is homomorphism and $r$ is an idempotent, therefore each $s_i$ is an idempotent in $0^+(\ber^+)$. But $\widetilde{\phi}(r) = \widetilde{\phi}(r+q)= \widetilde{\phi}(r) + \widetilde{\phi}(q) = \begin{pmatrix}
    s_1\\
    \vdots\\
    s_u
\end{pmatrix} + \begin{pmatrix}
    p\\
    \vdots\\
    p
\end{pmatrix} = \begin{pmatrix}
    s_1 + p\\
    \vdots\\
    s_u + p
\end{pmatrix}$. Now for each $i \in \{1,2,\cdots, u\}$, $s_i + p \in L$ and since for each $i =1,2,\cdots,u, s_i+p$ is an idempotent in $0^+(\ber^+)$ so $s_i+p \in E(L)$. Therefore, $\widetilde{\phi}(r) \in (E(L))^u$, i.e., $r \in \widetilde{\phi}^{-1}[((E(L))^u] \subseteq \widetilde{\phi}^{-1}[C\ell_{(0^+(\ber^+))^u}{C}^u]=Cl_{0^+((\ber ^+)^v)} \phi^{-1}[C^u]$. Therefore, $\phi^{-1}[C^u]$ is an $SC^*$ set near zero in $(\ber ^+)^v$. 
\end{proof}

\begin{theorem}
    Let $u,v \in \ben$ and $M$ be a $u \times v $ matrix with entries from $\ber^+ \cup \{0\}$ which is $IPR/\ber^+$. If $C$ is a thick set near zero in $\ber^+$, then $\{\vec{x} \in (\ber^+)^v: M\vec{x} \in C^u\}$ is also thick near zero in $(\ber^+)^v$.
\end{theorem}

\begin{proof}
     Let $\phi : (\ber ^+)^v \longrightarrow (\ber^+) ^u$ be defined by $\phi(\vec {x}) = M\vec {x}$ and let $\widetilde{\phi}: \beta ((\ber^+)^v)\longrightarrow (\beta \ber^+)^u$ be its continuous extension. Let $C$ be a thick set near zero in $\ber ^+$. Therefore, there exists a minimal left ideal $L$ of $0^+(\ber^+)$ with $L \subseteq C\ell_{0^+(\ber^+)}{C}$. Let $p$ be an idempotent in $L$.  Let $\vec{p}= \begin{pmatrix}
   p\\
   \vdots\\
    p
\end{pmatrix} \in (0^+(\ber^+))^u$. Since $M$ is $IPR/\ber^+$, by Lemma \ref{leA} there exists an idempotent $q \in 0^+((\ber^+)^v) $ such that $\widetilde{\phi}(q)=\vec{p}$. We claim that $0^+((\ber^+)^v) + q \subseteq C\ell_{0^+((\ber^+)^v)} \phi^{-1}[C^u]$. Let $r \in 0^+((\ber^+)^v)$ and let $\widetilde{\phi}(r) =\begin{pmatrix}
    s_1\\
    \vdots\\
    s_u
\end{pmatrix}$. Then by Lemma \ref{l1}, $\widetilde{\phi}(r) \in (0^+(\ber^+))^u$. Now,  $\widetilde{\phi}(r+q) = \widetilde{\phi}(r)+ \widetilde{\phi}(q) =\widetilde{\phi}(r) + \vec{p}$ = $\begin{pmatrix}
    s_1\\
    \vdots\\
    s_u
\end{pmatrix} + \begin{pmatrix}
    p\\
    \vdots\\
    p
\end{pmatrix} = \begin{pmatrix}
    s_1+p\\
    \vdots\\
    s_u+p
\end{pmatrix} \in L^u$. Therefore, $\widetilde{\phi}(r+q) \in L^u \subseteq C\ell_{(0^+(\ber^+))^u}{C}^u$, i.e.,$r+q \in \widetilde{\phi}^{-1}[C\ell_{(0^+(\ber^+))^u}{C}^u]= C\ell_{0^+((\ber^+)^v)} \; \phi^{-1}[C^u]$.

\end{proof}

\begin{theorem}
    Let $u,v \in \ben$ and $M$ be a $u \times v $ matrix with entries from $\ber^+ \cup \{0\}$ which is $IPR/\ber^+$. If $C$ is a $PS^*$ set near zero in $\ber^+$, then $\{\vec{x} \in (\ber^+)^v: M\vec{x} \in C^u\}$ is a $PS^*$ set near zero in $(\ber^+)^v$.
\end{theorem}

\begin{proof}
     Let $\phi : (\ber ^+)^v \longrightarrow (\ber^+) ^u$ be defined by $\phi(\vec {x}) = M\vec {x}$ and let \sloppy $\widetilde{\phi}: \beta ((\ber^+)^v)\longrightarrow (\beta \ber^+)^u$ be its continuous extension. Let $C$ be a $PS^*$ set near zero in $\ber ^+$. Therefore, $K(0^+(\ber^+)) \subseteq C \ell_{0^+(\ber^+)} {C}$. Let $p$ be an idempotent in $K(0^+(\ber^+))$. Then by Theorem \ref{th4}, for every $P \in p, \epsilon >0$, there exists $\vec {x} \in (0, \epsilon)^v$ such that  $M\vec{x} \in P^u.$ By Lemma \ref{leA} , $\widetilde{\phi}(K(0^+((\ber^+)^v))) =\widetilde{\phi}(0^+((\ber^+)^v)) \bigcap (K(0^+(\ber^+)))^u$. Thus, $\widetilde{\phi}(K(0^+((\ber^+)^v))) \subseteq  (K(0^+(\ber^+)))^u \subseteq C\ell_{(0^+(\ber^+))^u}{C}^u.$
     Therefore, $K(0^+((\ber^+)^v)) \subseteq \widetilde{\phi}^{-1}(C\ell_{(0^+(\ber^+))^u}{C}^u) = C\ell_{0^+((\ber^+)^v)}\phi^{-1}[C^u]$. Thus, $\phi^{-1}[C^u]$ is a $PS^*$ set near zero.

\end{proof}

\begin{theorem}
    Let $u,v \in \ben$ and $M$ be a $u \times v $ matrix with entries from $\ber^+ \cup \{0\}$ and with no row of $M$ as the zero row . Let $\Psi$ be $Q^*$ set near zero or $P^*$ set near zero in $\ber^+$ . If $C$ is a $\Psi$ set in $\ber^+$, then $\{\vec{x} \in (\ber^+)^v: M\vec {x} \in C^u\}$ is a $\Psi$ set in $(\ber^+)^v$.
\end{theorem}

\begin{proof}
     Let $\phi : (\ber ^+)^v \longrightarrow (\ber^+) ^u$ be defined by $\phi(\vec {x}) = M\vec {x}$ and let $\widetilde{\phi}: \beta ((\ber^+)^v)\longrightarrow (\beta \ber^+)^u$ be its continuous extension.\\
      $\underline{\text{Case} \; \Psi = \text{Q}^* \text{near zero}}$.  Let $C$ be a Q$^*$ set near zero in $\ber ^+$. Let $B=\{\vec{x} \in (\ber^+)^v: M\vec {x} \in C^u\}$. Suppose $B$ is not $Q^*$ set near zero in $(\ber^+)^v$. Then $(\ber^+)^v \setminus B$ is $Q$ set near zero. So there exists a sequence $\{\vec{s}_n\}_{n=1}^\infty$ in $(\ber^+)^v$ such that $\sum_{n=1}^\infty \vec{s}_n$ converges and such that whenever $m<n, \vec{s}_n \in \vec{s}_m+ ((\ber^+)^v \setminus B)$. We write $\vec{s}_n=\begin{pmatrix}
          s_{n1}\\
          \vdots\\
          s_{nv}
      \end{pmatrix}$, $\vec{s}_m=\begin{pmatrix}
          s_{m1}\\
          \vdots\\
          s_{mv}
      \end{pmatrix}$. Given, whenever $m<n, \vec{s}_n \in \vec{s}_m+ ((\ber^+)^v \setminus B)$. So, $\vec{s}_n - \vec{s}_m \notin B$, i.e., $\phi(\vec{s}_n - \vec{s}_m) \notin C^u$.i.e., $M(\vec{s}_n - \vec{s}_m) \notin C^u$.i.e., there exists $k_{m,n} \in \{1,2, \cdots,u\}$ such that the $k_{m,n}^{th}$ entry of the product matrix $M(\vec{s}_n - \vec{s}_m)$ does not belong to $C$. By Ramsey's Theorem, if the collection of all the subsets of $\ben$ of cardinality 2 is $u$ coloured, there exists an infinite subset $A$ of $\ben$ such that the collection of all subsets of $A$ of cardinality 2 is monochromatic. Therefore, $\{\{m,n\}: m,n \in A\}$ and in particular, $\{\{m,n\}: m,n \in A, m < n\}$ is monochromatic. Choose $k_{m,n} =i$ where $i \in \{1,2,\cdots,u\}$. Let $A=\{t(n)\}_{n=1}^\infty.$ Let for $n \in \ben, x_n=\sum_{j=1}^va_{ij}s_{t(n)j}$ where  $M=(a_{ij})$. Then $\sum_{n=1}^\infty x_n$ converges and whenever $m<n$ in $\ben, x_n - x_m \notin C$,i.e., $x_n -x_m \in \ber^+ \setminus C$ whenever $m<n$. This is a contradiction since $C$ is a $Q^*$ set near zero. \\
      $\underline{\text{Case} \; \Psi = \text{P}^* \text{near zero}}$. Let $C$ be a $P^*$ set near zero in $\ber^+$. Therefore, $\ber^+ \setminus C$ is not a P set near zero. Therefore, there exists $ k \in \ben,$ and $ \epsilon' >0$ such that $(\ber^+ \setminus C) \cap (0,\epsilon')$ does not contain any length $k$ Arithmetic progression(AP).
      
      For $k,u$, by Vander Waerden's theorem, there exists $m \in \ben$ such that whenever $\{1,2,\cdots,m\}$ is $u$ coloured, there is a monochromatic length $k$ Arithmetic Progression. Let $B=\{\vec{x} \in (\ber^+)^v: M\vec {x} \in C^u\}$. Suppose $B$ is not $P^*$ set near zero in $(\ber^+)^v$. Then $(\ber^+)^v \setminus B$ is $P$ set near zero. Let $M=(a_{ij})$ and $a=\text{max}\{a_{ij}| 1 \le i \le u, 1 \le j \le v \} $. Therefore, for $\frac{\epsilon'}{av} >0, m \in \ben$, there exists $ \vec{s}, \vec{d} \in (\ber^+)^v$ such that $\{\vec{s}+\vec{d}, \cdots, \vec{s}+m\vec{d}\} \subseteq ((\ber^+)^v \setminus B) \cap (0,\frac{\epsilon'}{av})^v$, where $\vec{s}= \begin{pmatrix}
          s_1\\
          \vdots\\
          s_v
      \end{pmatrix}$, $\vec{d}= \begin{pmatrix}
          d_1\\
          \vdots\\
          d_v
      \end{pmatrix}$, i.e., for each $t \in \{1,2, \cdots, m\}, M(\vec{s}+t\vec{d}) \notin C^u$,i.e., for each $t \in \{1,2,\cdots, m\}$ there exists an entry, $i(t)^{th}$ entry,  of  $M(\vec{s}+t\vec{d})$ which does not belong to $C$. By Vander Waerden's theorem, whenever $\{1,2, \cdots ,m\}$ is $u$ coloured, there exists $r \in \{1,2,\cdots,u\}$, $t',f \in \{1,2,\cdots,m\}$ such that $\{t'+f, \cdots, t'+kf\}$ belongs to the $r^{th}$ cell where $t'+kf \le m$. Define $i: \{1,2,\cdots,m\} \longrightarrow \{1,2,\cdots, u\}$ such that $i(t'+f)=\cdots=i(t'+kf)=r$ where $t'+f, \cdots, t'+kf \in \{1,2,\cdots, m\}$. Therefore, for each $t'' \in \{t'+f, \cdots, t'+kf\}$, $\sum_{j=1}^va_{rj}(s_j+t''d_j) \notin C$. Let $b=\sum_{j=1}^va_{rj}s_j+t'\sum_{j=1}^va_{rj}d_j, e =f\sum_{j=1}^va_{rj}d_j$. then $b+e \notin C, b+2e \notin C$ and so on. Also, $b+e=\sum_{j=1}^va_{rj}s_j+(t'+f)\sum_{j=1}^va_{rj}d_j < a\frac{\epsilon'}{av}+\cdots + a\frac{\epsilon'}{av}=\epsilon'$. Similarly, $b+2e \in (0,\epsilon')$ and so on. Therefore, $\{b+e,b+2e,\cdots,b+ke\} \subseteq (\ber^+ \setminus C) \cap (0,\epsilon')$. This is a contradiction since  $(\ber^+ \setminus C) \cap (0,\epsilon')$ does not contain any length $k$ Arithmetic progression. Thus, $B$ is a $P^*$ set near zero in $(\ber^+)^v$.\\
    
\end{proof}

\begin{theorem}
    Let $u,v \in \ben$ and $M$ be a $u \times v $ matrix with entries from $\ber^+ \cup \{0\}$ which is $IPR/\ber^+$. Let $\Psi$ be central$^*$ near zero or $QC^*$ near zero in $\ber^+$. If $C$ is a $\Psi$ set in $\ber^+$, then $\{\vec{x} \in (\ber^+)^v: A\vec {x} \in C^u\}$ is a $\Psi$ set in $(\ber^+)^v$.
\end{theorem}

\begin{proof}
     Let $\phi : (\ber ^+)^v \longrightarrow (\ber^+) ^u$ be defined by $\phi(\vec {x}) = M\vec {x}$ and let $\widetilde{\phi}: \beta ((\ber^+)^v)\longrightarrow (\beta \ber^+)^u$ be its continuous extension.\\
     $\underline{\text{Case} \; \Psi = \text{central}^* \text{near zero}}$. Let $C$ be a central$^*$ set near zero in $\ber ^+$. Therefore, $E(K(0^+(\ber^+))) \subseteq C\ell_{0^+(\ber^+)}{C}$. Let $p$ be an idempotent in $K(0^+(\ber^+))$. Then by Theorem \ref{th4}, for every $P \in p, \epsilon >0$ there exists $\vec {x} \in (0, \epsilon)^v$ such that  $M\vec {x} \in P^u.$ By Lemma \ref{leA} , $\widetilde{\phi}(K(0^+((\ber^+)^v))) =\widetilde{\phi}(0^+((\ber^+)^v)) \bigcap (K(0^+(\ber^+)))^u$. Let $q$ be an idempotent in $K(0^+((\ber^+)^v))$. Then $\widetilde{\phi}(q)$ is an idempotent in $(K(0^+(\ber^+)))^u$. So, $\widetilde{\phi}(q) \in C\ell_{(0^+(\ber^+))^u}{C}^u$. Therefore, $q \in \widetilde{\phi}^{-1}[C\ell_{(0^+(\ber^+))^u}{C}^u] =  C\ell_{0^+((\ber^+))^v}{\phi}^{-1}[C^u]$.\\
      $\underline{\text{Case} \; \Psi = \text{QC}^*\text{near zero}}$.  Let $C$ be a QC$^*$ set near zero in $\ber ^+$. Therefore, \sloppy $E(C\ell_{0^+(\ber^+)} K(0^+(\ber^+))) \subseteq C\ell_{0^+(\ber^+)}{C} $. We will prove that   $E(C\ell_{(0^+(\ber^+))^v}(K(0^+((\ber ^+))^v)) \subseteq C\ell_{0^+((\ber^+)^v)} {\phi^{-1}[C^u]} $. Let $p \in E(C \ell_{0^+(\ber^+)} (K(0^+(\ber^+)))$. Thus, \\
          $\widetilde{\phi}[C \ell_{0^+((\ber^+)^v)} \; K(0^+((\ber^+)^v))] \subseteq C\ell_{(0^+(\ber^+))^u} \; \widetilde{\phi}[K(0^+((\ber^+)^v))]$ \\
    $ = C\ell_{(0^+(\ber^+))^u} \;[ \widetilde{\phi}(0^+((\ber^+)^v))\bigcap (K(0^+(\ber^+))^u)] \text{(Using Lemma \ref{leA})}$ \\
      $ \subseteq C\ell \;_{(0^+(\ber^+))^u}\widetilde{\phi}[0^+((\ber^+)^v)] \bigcap  C\ell_{(0^+(\ber^+))^u} (K(0^+(\ber^+)))^u$\\
       $=C\ell \;_{(0^+(\ber^+))^u}\widetilde{\phi}[0^+((\ber^+)^v)] \bigcap (C\ell_{0^+(\ber^+)} K(0^+(\ber^+)))^u$\\
     \sloppy Thus if $q$ is an idempotent in $C \ell_{0^+((\ber^+)^v)} [K(0^+((\ber^+)^v))]$, then $\widetilde{\phi}(q)$ is an idempotent in $(C\ell_{0^+(\ber^+)}\; K(0^+(\ber^+)))^u$ and $(E(C\ell_{0^+(\ber^+)}\; K(0^+(\ber^+))))^u \subseteq (C\ell_{0^+(\ber^+)}\; C)^u \subseteq C\ell_{(0^+(\ber^+))^u}{C}^u$. Therefore, $\widetilde{\phi}(q) \in (C\ell_{0^+(\ber^+)}C)^u$, i.e., $q \in \widetilde{\phi}^{-1}[C\ell_{(0^+(\ber^+))^u}{C}^u] \subseteq C \ell_{0^+((\ber^+)^v)}(\phi^{-1}[C^u])$.\\

\end{proof}

\begin{lemma} \label{leC}
   Let $u,v \in \ben$ and let $M$ be a $u \times v$ matrix with entries from $\ber^+ \cup \{0\}$ and with no row equal to the zero row. Define $\phi : (\ber ^+)^v \longrightarrow (\ber^+)^u$ by $\phi(\vec {x}) = M\vec {x}$ and let $\widetilde{\phi}: \beta ((\ber^+)^v)\longrightarrow (\beta \ber^+)^u$ be its continuous extension.
    Let $q \in J_0((\ber^+)^v), \widetilde{\phi}(q)= \begin{pmatrix}
    p_1\\
    \vdots\\
    p_u
\end{pmatrix}$  then for each $i \in \{1,2, \cdots, u\},$ $ p_i \in J_0(\ber^+).$   
\end{lemma}
\begin{proof}
  Let $i \in \{1,2, \cdots, u\}$ and let $B \in p_i$. We will show that $B$ is a J set  near zero. Let $F \in \mathcal{P}_f(\mathcal{R}_0^+)$. We will show that for every $\delta>0$, there exists $a \in (0,\delta)$ and $H \in \mathcal{P}_f(\ben)$ such that for all $f\in F, a + \sum_{t \in H}f(t) \in B$. Recall that $J_0(\ber^+)=\{ p \in 0^+(\ber^+): \forall \; B \in p, B$ is a J set near zero\}. By Lemma \ref{l1}, $\widetilde{\phi}(q) \in (0^+(\ber^+))^u$, so for each $i \in \{1,2,\cdots,u\}$, $ p_i \in 0^+(\ber^+)$. Let $\pi_i: (\beta (\ber^+))^u \longrightarrow \beta \ber^+$ be the $i^{th}$ projection map, then $\pi_i^{-1}[C\ell_{0^+(\ber^+)}{B}]$ is a neighbourhood of $\widetilde{\phi}(q)$ so there exists a neighbourhood, say, $C\ell_{\beta (\ber^+)^v}{D}$ of $q$ such that $\widetilde{\phi}(C\ell_{(\beta (\ber^+)^v}{D}) \subseteq \pi_i^{-1}[{C\ell_{0^+(\ber^+)}}{B}]$, i. e., there exists $D \in q$ such that $(\pi_i \circ \widetilde{\phi})(C\ell_{\beta (\ber^+)^v}{D}) \subseteq C\ell_{0^+(\ber^+)}{B}$ , i.e., there exists $D \in q$ such that $(\pi_i \circ \phi)(D) \subseteq B$. Let $ M=(a_{ij}) , \sum_{j=1}^v a_{ij} =k$ Now,  $D$ is a J-set near zero in ($\ber^+)^v$. So let $F'=\{ \frac{1}{k}\begin{pmatrix}
      f(n)\\
      \vdots\\
      f(n)
  \end{pmatrix} = \vec{f}(n)| f \in F \}$, then $F'\in \mathcal{P}_f((\mathcal{R}_0^+)^v))$. Let $\epsilon > 0$, then, for $\frac{\epsilon}{k} > 0, F' \in  \mathcal{P}_f((\mathcal{R}_0^+)^v))$, there exists $\vec{b} \in (0, \frac{\epsilon}{k})^v, H \in  \mathcal{P}_f(\ben)$ such that for all $\vec{f} \in F', \vec{b} + \sum_{t \in H} \vec{f}(t) \in D$, i.e., there exists $\vec{b} \in (0, \frac{\epsilon}{k})^v, H \in  \mathcal{P}_f(\ben)$ such that for all $\vec{f} \in F', (\pi_{i} \circ \phi) (\vec{b} + \sum_{t \in H} \vec{f}(t) \in B$. Let $\vec{b}= \begin{pmatrix}
      b_1\\
      \vdots\\
      b_v
  \end{pmatrix}$. Then $\sum_{j=1}^v a_{ij}(b_j + \sum_{t \in H} \frac{f(t)}{k}) \in B$. Let  $\sum_{j=1}^v a_{ij}b_j =c$ then $c \in (0,\epsilon)$. Thus $c + \sum_{t \in H} f(t) \in B$ and hence $B$ is a J-set near zero.
\end{proof}
\begin{theorem}
    Let $u,v \in \ben$ and $M$ be a $u \times v $ matrix with entries from $\ber^+ \cup \{0\}$ and with no row of $M$ as the zero row. Let $\Psi$ be C$^*$ near zero or $J^*$ set near zero or $IP^*$ near zero. If $C$ is a $\Psi$ set in $\ber^+$, then $\{\vec{x} \in (\ber^+)^v: M\vec {x} \in C^u\}$ is a $\Psi$ set in $(\ber^+)^v$.
\end{theorem}

\begin{proof}
   Let $\phi : (\ber ^+)^v \longrightarrow (\ber^+) ^u$ be defined by $\phi(\vec {x}) = M\vec {x}$ and let $\widetilde{\phi}: \beta ((\ber^+)^v)\longrightarrow (\beta \ber^+)^u$ be its continuous extension.\\
     $\underline{\text{Case} \; \Psi = C^* \text{near zero}}$. Let $C$ be a $C^*$ set near zero in $\ber ^+$. Therefore, $E(J_0(\ber^+)) \subseteq C\ell_{0^+(\ber^+)}{C}$. Let $q$ be an idempotent in $J_0((\ber^+)^v))$ and let $\widetilde{\phi}(q)= \begin{pmatrix}
         p_1\\
         \vdots\\
         p_u
     \end{pmatrix}$.  Then by Lemma \ref{leC}, each $p_i$ is an idempotent in $J_0(\ber^+)$, therefore, $p_i \in E( J_0(\ber^+)) \subseteq C\ell_{0^+(\ber^+)}{C}$. Therefore, $\widetilde{\phi}(q)  \in C\ell_{(0^+(\ber^+))^u}{C}^u$, thus, $q \in \widetilde{\phi}^{-1}(C\ell_{(0^+(\ber^+)^u}{C}^u)= C\ell_{0^+((\ber^+)^v)} \phi^{-1}[C^u]$.\\
     $\underline{\text{Case} \; \Psi = J^* \text{near zero}}$. Let $C$ be a $J^*$ set near zero in $\ber ^+$. Therefore, $J_0(\ber^+) \subseteq C\ell_{0^+(\ber^+)}{C}$. Let $q \in J_0((\ber^+)^v)$ and let $\widetilde{\phi}(q)= \begin{pmatrix}
         p_1\\
         \vdots\\
         p_u
     \end{pmatrix}$.  Then by Lemma \ref{leC}, each $p_i \in J_0(\ber^+)$, therefore, $p_i \in  C\ell_{0^+(\ber^+)}{C}$. Therefore, $\widetilde{\phi}(q)  \in C\ell_{(0^+(\ber^+))^u}{C}^u$, thus, $q \in \widetilde{\phi}^{-1}[C\ell_{(0^+(\ber^+))^u}{C}^u]= C\ell_{0^+((\ber^+)^v)} \phi^{-1}[C^u]$.
     $\underline{\text{Case} \; \Psi = \text{IP}^* \text{near zero}}$.  Let $C$ be a IP$^*$ set near zero in $\ber ^+$. Then, $E(0^+(\ber^+)) \subseteq C\ell_{0^+(\ber^+)}{C} $. Let $q$ be an idempotent in $0^+((\ber^+)^v)$. and let $\widetilde{\phi}(q)= \begin{pmatrix}
         p_1\\
         \vdots\\
         p_u
     \end{pmatrix}$.  Then by Lemma \ref{l1}, $\widetilde{\phi}(q) \in (0^+(\ber^+))^u$. Now, each $p_i$ is an idempotent in $0^+(\ber^+)$, therefore, $p_i \in  C\ell_{0^+(\ber^+)}{C}$, i.e., $\widetilde{\phi}(q)  \in C\ell_{(0^+(\ber^+))^u}{C}^u$, thus, $q \in \widetilde{\phi}^{-1}[C\ell_{(0^+(\ber^+))^u}{C}^u]= C \ell_{0^+((\ber^+)^v)}(\phi^{-1}[C^u]).$
\end{proof}

\bibliographystyle{amsplain}

\end{document}